\documentclass{amsart}
\usepackage{amssymb}
\usepackage[pagebackref]{hyperref}
\usepackage{tikz}

\title{Vieta jumping and small norms in quadratic number fields}
\author{Franz Lemmermeyer}
\address{M\"orikeweg 1, 73489 Jagstzell, Germany}
\email{franz.lemmermeyer@gmx.de}

\newtheorem{lemma}{Lemma}
\newtheorem{prop}[lemma]{Proposition}
\newtheorem{theorem}[lemma]{Theorem}

\newcommand{\rsp}{\raisebox{0em}[2.7ex][1.3ex]{\rule{0em}{2ex} }}

\newcommand{\N}{{\mathbb N}}
\newcommand{\Z}{{\mathbb Z}}
\newcommand{\Q}{{\mathbb Q}}

\newcommand{\lra}{\longrightarrow}

\newcommand{\lrs}{\stackrel{\sharp}{\lra}}
\newcommand{\lrf}{\stackrel{\flat}{\lra}}

\newcommand{\eps}{\varepsilon}

\newcommand{\cC}{{\mathcal C}}
\newcommand{\cM}{{\mathcal M}}
\newcommand{\cO}{{\mathcal O}}
\newcommand{\cP}{{\mathcal P}}
\newcommand{\OO}{{\mathcal O}}
\newcommand{\fp}{\mathfrak p}

\begin{document}

\begin{abstract}
  In this article we explain the connection between the famous problem
  from the IMO 1988 and elements of small norms in quadratic number
  fields with parametrized units.
\end{abstract}

\maketitle

\section*{Introduction}
The following problem\footnote{See \cite{LQNF}.} was posed
at the International Mathematical Olympiad in 1988. 

\begin{quote}
  {\em Let $a$ and $b$ be positive integers such that $ab + 1$ divides
    $a^2 + b^2$. Prove that $\frac{a^2 + b^2}{ab+1}$ is a perfect square.}
\end{quote}

If we set $\frac{a^2 + b^2}{ab+1} = k$, then the statement is
equivalent to the claim that if $(a,b)$ is an integral point on the
conic
\begin{equation}\label{EC0}
  \cC_k: a^2 - kab + b^2 = k,
\end{equation}
then $k$ is a perfect square.

In this article we will give several proofs of this result,
point out connections with elements of small numbers in quadratic
number fields of Richaud-Degert type, and present various generalizations.
It is our intention to show that Vieta jumping is a technique that may
be used outside of olympiad problems.

\section{Vieta Jumping and an IMO Problem}

Before we look at the solution of this IMO problem we determine all solutions
in integers and rational numbers when the quotient is a perfect square.
If $\frac{a^2 + b^2}{ab+1} = m^2$, then $(a,b)$ is an integral point
on the curve
\begin{equation}\label{EC1}
  \cC_{m^2}: a^2 - m^2 ab + b^2 = m^2.
\end{equation}

For finding all rational points on $\cC_{m^2}$ we compute the points of
intersection of $\cC_{m^2}$ and the lines $\ell_t: y = t(x-m)$ through the
point $(m,0)$ and find $(x,y)$ with
\begin{equation}\label{Prat}
  x = \frac{m(t^2-1)}{t^2 - m^2t + 1}, \quad
  y = \frac{m(m^2t^2 - 2t)}{t^2 - m^2t + 1}
\end{equation}
as the second point of intersection $(x,y)$ of $\ell_t$ with $\cC_{m^2}$.

\begin{prop}
  The rational points $(x,y) \ne (m,0)$ on the conic $\cC_{m^2}$ are
  parametrized by (\ref{Prat}); the point $(m,0)$ corresponds to $t = \infty$.
\end{prop}

Finding all integral solutions of (\ref{EC1}) is more difficult.  The
obvious integral solutions are $(a,b) = (\pm m, 0)$ and $(0, \pm m)$.
The idea is constructing new integral points from known ones using
``Vieta jumping'', and then reversing this procedure and showing that
each integral point is generated by the obvious points above.

\begin{prop}\label{Pr1}
  The integral points on the conic $\cC_{m^2}$ are exactly the 
  points in the sequences
\begin{align}
  \label{Seq1}
  \ldots,\ (-m,-m^3), \ (-m,0), \ & (m,0), \ (m,m^3), \ (m^5-m,m^3), \ \ldots \\
  \label{Seq2}
  \ldots,\ (-m^3,-m), \ (0,-m), \ & (0,m), \ (m^3, m), (m^3,m^5-m), \ \ldots.
\end{align}
These points have the form $(a_n,a_{n+1})$ and $(a_{n+2},a_{n+1})$,
where the sequence $(a_n)$ satisfies the recurrence relation
$$ a_{n+2} = m^2a_{n+1} - a_n $$
and has the property $a_na_{n+2} = a_{n+1}^2 - m^2$.
\end{prop}

\begin{figure}[ht!]
\begin{tikzpicture}[scale=0.6]
  \clip (-2.8,-3.2) rectangle (8.5,8.5);  
  \draw[->] (-3.2,0) -- (8.5,0);
  \draw[->] (0,-3.2) -- (0,8.5);  
  \foreach \x in {1,2,3,4,5,6,7,8}
  \draw (\x cm,1pt) -- (\x cm,-1pt) node[anchor=north]
        {{\small $\x$}};  
  \foreach \y in {1,2,3,4,5,6,7,8}
  \draw (1pt,\y cm) -- (-1pt,\y cm) node[anchor=east]
        {{\small $\y$}};
  \draw (-2,0) node[anchor=south] {{\small $-2$}};
  \draw (0,-2) node[anchor= west] {{\small $-2$}};
  \draw[line width=0.7pt] plot[variable=\t, domain=0.28:2,samples=101]
  ({(2*\t^2-2)/(\t^2-4*\t+1)},{(8*\t^2-4*\t)/(\t^2-4*\t+1)});
  \draw[line width=0.7pt] plot[variable=\t, domain=0.28:2.2,samples=101]
  ({-(2*\t^2-2)/(\t^2-4*\t+1)},{-(8*\t^2-4*\t)/(\t^2-4*\t+1)});
  \fill ( 2,0) circle(2.5pt);
  \fill (-2,0) circle(2.5pt);
  \fill (0,-2) circle(2.5pt);
  \fill (0, 2) circle(2.5pt);
  \fill (8, 2) circle(2.5pt);
  \fill (2, 8) circle(2.5pt);
  \draw [line width=0.8pt, dotted] (-2,0) -- (2,0);
  \draw [line width=0.8pt, dotted] (0,-2) -- (0,2);
  \draw [line width=0.8pt, dotted] (2,0) -- (2,8);
  \draw [line width=0.8pt, dotted] (0,2) -- (8,2);
  \draw [line width=0.8pt, dotted] (-3.2,-2) -- (0,-2);
  \draw [line width=0.8pt, dotted] (-2,-3.2) -- (-2,0);
  \draw [line width=0.8pt, dotted] (2,8) -- (8.4,8);
  \draw [line width=0.8pt, dotted] (8,2) -- (8,8.4);
\end{tikzpicture}
 \caption{Vieta jumping on $\cC_4: x^2 - 4xy + y^2 = 4$.}\label{AbbVieta} 
\end{figure}
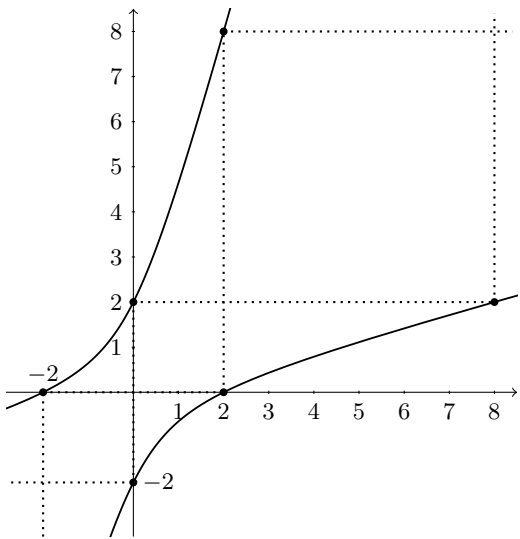

If we begin with the point $(m,0)$ and set $a = m$ in (\ref{EC1}), then
we obtain the quadratic equation $m^2 - m^3b + b^2 = m^2$, i.e.,
$0 = b^2 - m^3b = b(b-m^3)$ with the known solution $b_1 = 0$ and the
second solution $b_2 = m^3$, giving rise to the point $(m, m^3)$.
If we now set $b = m^3$, we obtain a quadratic equation in $a$ with
one integral solution $a_1 = m$; the second solution is necessarily
integral:

\begin{lemma}\label{LemInt}
  If the equation $a^2 - m^2ab + b^2 = m^2$ has a solution $a_1  \in \Z$
  when considered as a quadratic equation in $a$, then the second solution
  $a_2$ is also integral.
\end{lemma}

\begin{proof}
  Write $a^2 - m^2ab + b^2 - m^2 = (a-a_1)(a-a_2)$; then
  $a_1 + a_2 = m^2$ implies that $a_2$ is an integer whenever $a_1$ is.  
\end{proof}

Continuing in this way we find an infinite sequence of integral
points on the conic $\cC_m$. Assume that we have an integral point 
$(a,b)$ on $\cC_{m^2}$. For computing a second point with the same
$y$-coordinate $b$ we have to solve
$$ x^2 - m^2xb + b^2 = m^2 $$
knowing that
$$ a^2 - m^2ab + b^2 = m^2. $$
Subtracting these equations from each other we find
$$ 0 = x^2 - a^2 - m^2b(x-a) = (x-a)(x+a-m^2b), $$
hence the solutions of the quadratic equation are the known solution
$x_1 = a$ as well as $x_2 = m^2b - a$. Thus Vieta jumping produces the
new solution
$$ \flat(a,b) = (m^2b-a,b). $$

Similarly, if we keep the $x$-coordinate fixed we have to solve
$a^2 - m^2ay + y^2 = m^2$, and this produces the new point
$$ \sharp(a,b) = (a,m^2a-b). $$

Thus we obtain the following operations (see Fig.~\ref{descent}) for
finding new integral points on $\cC_k$; points $(a,b)$ with $a < b$
are on the upper branch of the curve, those with $a > b$ lie on the
lower branch. Starting with the obvious points $(\pm m, 0)$ and
$(0, \pm m)$ we now can construct the two sequences of integral points
(\ref{Seq1}) and (\ref{Seq2}) on $\cC_{m^2}$.

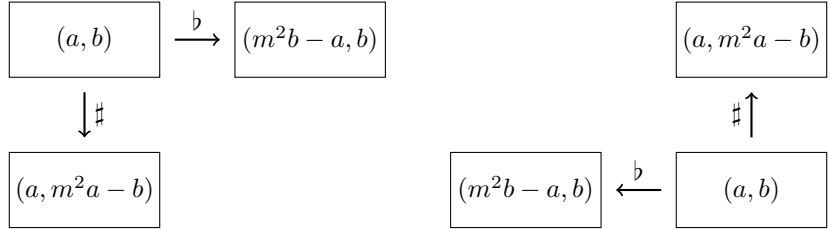
\begin{figure}[ht!]
\begin{tikzpicture}
  \draw (0,0) rectangle (2,1);
  \draw (3,2) rectangle (5,3);
  \draw (0,2) rectangle (2,3);
  \draw[->,line width=0.8pt] (1,1.8) -- (1,1.2) node[midway,right] {$\sharp$};
  \draw[->,line width=0.8pt] (2.2,2.5) -- (2.8,2.5)
        node[midway, above] {$\flat$};
  \draw (1,0.5) node {$(a,m^2a - b)$};
  \draw (1,2.5) node {$(a,b)$};
  \draw (4,2.5) node {$(m^2b - a,b)$};
\end{tikzpicture} \qquad \begin{tikzpicture}
  \draw (0,0) rectangle (2,1);
  \draw (-3,0) rectangle (-1,1);
  \draw (0,2) rectangle (2,3);
  \draw[->,line width=0.8pt] (1,1.2) -- (1,1.8)
         node[midway,left] {$\sharp$};
  \draw[->,line width=0.8pt] (-0.2,0.5) -- (-0.8,0.5)
         node[midway, above] {$\flat$};
  \draw (1,0.5) node {$(a,b)$};
  \draw ( 1,2.5) node {$(a,m^2a - b)$};
  \draw (-2,0.5) node {$(m^2b - a,b)$};
\end{tikzpicture} 
\caption{Vieta jumping on $\cC_m: a^2 - m^2 ab + b^2 = m^2$ for points
on the lower (left) and upper branch (right).}\label{descent}
\end{figure}

Next we will show that each integral point on $\cC_{m^2}$ is among those in 
(\ref{Seq1}) or (\ref{Seq2}). The basic idea is using descent: Start
with an arbitrary integral point $(A, B)$ and use Vieta jumping to find
an integral point close to the origin; the only such points with coordinates
whose absolute values are $\le m$ are the points $(\pm m, 0)$ and $(0,\pm m)$,
and reversing the procedure we see that $(A,B)$ belongs to one of the two
sequences.

The tool for proving this claim is the following lemma:

\begin{lemma}\label{L2}
  If $(a,b)$ is an integral point with positive coordinates $a \ne b$
  on $\cC_k$, then there is another integral point $(a', b')$ on $\cC_k$
  with $a', b' \ge 0$ and $a' + b' < a + b$.
\end{lemma}

\begin{proof}
  We set
  $$ (a',b') = \begin{cases}
    (a, ka-b) & \text{ if } b > a, \\
    (kb-a, b) & \text{ if } b < a. \\
  \end{cases} $$
  From $a^2 - kab + b^2 = k$ we deduce that
  $ka - b = \frac{a^2-k}b$ and $kb-a = \frac{b^2-k}a$. 
  \begin{itemize}
  \item  $a' + b' < a + b$: In the first case,
    $b' = ka - b = \frac{a^2-k}b < b$ since $a^2 - k < a^2 < b^2$;
    in the second case, $a' = kb-a  = \frac{b^2-k}a < a$ since
    $b^2 - k < b^2 < a^2$. Thus in both cases $a' + b' < a + b$.
  \item $a' \ge 0$, $b' \ge 0$: In fact, if
    $b' = ka-b < 0$ in the first case, then $ab' \le -1$, hence
    $$ 0 = {b'}^2 - kab' + a^2 - k \ge {b'}^2 + k + a^2 - k
         = a^2 + {b'^2} > 0, $$
    which is a contradiction; thus $b' \ge 0$.
  \end{itemize}  
\end{proof}

Next we show

\begin{lemma}
  For all integers $k \ge 1$, the conic $\cC_k$ does not contain any
  integral points $(a,b)$ with $a = b$ except when $k = 1$.
\end{lemma}

\begin{proof}
  If $a = b$, then the quotient
  $k = \frac{a^2 + b^2}{ab+1} = \frac{2a^2}{a^2+1}$ must be an integer;
  since $\gcd(2a^2, a^2+1) = \gcd(2,a^2+1)$ is either $1$ or $2$,
  this is only possible if $a = 0$ or $a = 1$. But $(0,0)$ is a point
  on $\cC_k$ if and only if $k = 0$, which we have excluded, and
  $(1,1)$ is a point on $\cC_k$ if and only if $k = 1$.
\end{proof}

Now we can finish the proof of Prop.~\ref{Pr1}:
Given an integral point $(a,b)$ in the first quadrant with
$a > 0$ and $b > 0$  we can find another integral point $(a',b')$ with
$a' \ge 0$, $b' \ge 0$ and $a' + b' < a + b$. We can repeat this step
until we end up with an integral point $(r,s)$ with $r = 0$ or $s = 0$.
But the only such points are $(\pm m, 0)$ and $(0, \pm m)$.

\subsection*{First Solution of the IMO problem}

Assume that $a^2 + b^2 = k(ab+1)$, which means that
$P(a, b)$ is an integral point on the conic $\cC_k: x^2 - kxy + y^2 = k$;
also assume that $k$ is not a square. Then $k \ge 2$ (and as a matter of
fact $k \ge 3$, since $k = 2$ implies $2 = (x-y)^2$, which is impossible
in integers).

We know from Lemma~\ref{L2} that as long as $a > 0$ and $b > 0$ we can find an
integral point $(a', b')$ on $\cC_k$ lying in the first quadrant with
$a' + b' < a + b$. Applying this step sufficiently often we obtain an
integral point of the form $(A,0)$ or $(0,A)$; but then $k = A^2$ is a
square.

\section{A family of equations}

Our next result strengthens an observation given without proof by
Milchev \cite{Milchev}:

\begin{theorem}\label{TM1}
  If the diophantine equation
  \begin{equation}\label{M1}
    x^2 - pxy + y^2 = q    
  \end{equation}
  has integral solutions for integers $p > 2$ and $0 < q \le p+1$, then
  $q$ is a square.
\end{theorem}

The result is best possible in the sense that if $q = p+2$, then there
are infinitely many integral solutions provided by Vieta jumping
starting with $(1,-1)$. Given an integral point $(a,b)$ on the conic
$\cM: x^2 - pxy + y^2 = p+2$, Vieta jumping provides us with
$\sharp(a,b) = (a,ap-b)$ and $\flat(a,b) = (bp-a,b)$, leading to the
chains of integral points
$$ \ldots (-p-1,-1) \lrf (1, -1) \lrs (1, p+1) \lrf (p^2+p-1,p+1) \lrs \ldots $$
and
$$ \ldots (-1,-p-1) \lrs (-1, 1) \lrf (p+1, 1) \lrs (p+1,p^2+p-1) \lrf \ldots.$$
Using the recurrence sequence $(a_n)$ defined by $a_1 = 1$,  $a_2 = p+1$
and $a_{n+1} = pa_n - a_{n-1}$, these points have the form $(a_n,a_{n+1})$ 
or  $(a_{n+1},a_n)$.

\begin{proof}[Proof fo Thm.~\ref{TM1}]
Assume that $(a,b)$ is an integral point with $\sqrt{q} < a < b$. Then
$\sharp(a,b) = (a,b')$ is another integral point, where
$b' = ap-b = \frac{a^2-q}{b}$. Clearly
$b' = \frac{a^2-q}{b} < \frac{a^2}q < \frac{ab}{b} = a$, and
$b' > 0$ since $a^2 > q$.

Thus as long as $\sqrt{q} < a < b$ we can find a new integral points
$(a,b')$ with $0 < b' < a$. If $b' > \sqrt{q}$, apply this construction
to $(b',a)$ and continue. This shows that if there is an integral point,
then there must be such a point $(a,b)$ with $0 < a \le \sqrt{q}$ and $b > 0$.
Applying the reduction step once more we find an integral point $(a,b)$
with $0 < a \le \sqrt{q}$ and $-\sqrt{q} \le b \le 0$. If $a = \sqrt{q}$,
$b = 0$ or $b = \sqrt{q}$, then $q$ is a square and we are done. If
the bounds are sharp, then $|a| \ge 1$ and $|b| \ge 1$, and we find
$ a^2 - pab + b^2 = a^2 + p|ab| + b^2 \ge 1 + p + 1 = p+2$. This
completes the proof.  
\end{proof}

If $p = 2$, then $q = x^2 - 2xy + y^2 = (x-y)^2$ has solutions if and
only if $q$ is a square. If $p = 1$, then $x^2 - xy + y^2 = q$
is equivalent to $(2x-y)^2 + 3y^2 = 4q$, and it is well known that
this equation is solvable if the prime factorization of $q$ contains
prime factors $\equiv 2 \bmod 3$ with an even exponent. The number of
closed paths depende on the number of ways $q$ can be written in the
form $q = c^2 + 3d^2$. The case $p = 0$ is similar: $x^2 + y^2 = q$
means $q$ is a sum of two squares, and this is the case if and only
if primes $p \equiv 3 \bmod 4$ occur an even number of times in the
prime factorization of $q$. If $q = A^2 + B^2$, then
$$ (A,B) \lrf (-A,B) \lrs (-A,-B) \lrf (A,-B) \lrs (A,B) $$
is a closed path with four points (only two if $A=0$ or $B = 0$), 
and the number of paths depends on the number of ways $q$ can be written
as a sum of two squares.

If $p < 0$, then clearly $q > 0$, and replacing $y$ by $-y$ transforms
$x^2 - pxy + y^2 = q$ into $x^2 - p'xy + y^2 = q$ with $p' > 0$, which
is covered by our results above.

The only case missing is $p \ge 3$ and $q < 0$.

\begin{theorem}\label{Thpq}
  The diophantine equation
  \begin{equation}\label{M2}
    x^2 - pxy + y^2 = q    
  \end{equation}
  does not have any integral solutions for integers $p > 0$ and
  $3-p \le q < 0$.
  
  If $q = 2-p$, then the integral solutions of (\ref{M2}) are given by
  $$ \ldots \lrf (1,1) \lrs (1,p-1) \lrf (p^2-p-1,p-1) \lrs
     (p^2-p-1,p^3 - p^2 - 2p + 1) \lrf \ldots $$
  or by the corresponding points in the third quadrant.
\end{theorem}

\begin{proof}   
  The minimal positive $x$- and $y$-coordinates of the conic are easily
  seen to be $\mu = \sqrt{\frac{-4q}{p^2-4}}$. Thus by Vieta jumping we can
  reduce each integral point on the conic to $(x,y)$ with
  $0 < x, y \le \mu$. If $q > 2-p$, then
  $$ \frac{-4q}{p^2-4} < \frac{4(p-2)}{p^2-4} = \frac4{p+2} \le 1 $$
  whenever $p \ge 2$. This implies that there is no integral point on the
  conic.  
\end{proof}

Theorem~\ref{Thpq} also has interesting corollaries such as the following:

\begin{prop}
  Assume that $x$ and $y$ are positive integers such that
  $3xy-1$ divides $x^2 + y^2$. Then
  $$ \frac{x^2 + y^2}{3xy - 1} = 1, $$
  and the only integral solutions have the form
  $(x,y) = (F_{2n-1}, F_{2n+1})$ and $(x,y) = (F_{2n+1}, F_{2n-1})$,
  where $F_n$ is the $n$-th Fibonacci number.
\end{prop}

\begin{figure}
  \begin{center}
      \begin{tikzpicture}
      \draw[->,opacity=0.4] (-0.5,0) -- (6.5,0);
      \draw[->,opacity=0.4] (0,-0.5) -- (0,6.5);
       \foreach \x in {1,2,3,4,5}
  \draw (\x cm,1pt) -- (\x cm,-1pt) node[anchor=north]
        {{\small $\x$}};  
  \foreach \y in {1,2,3,4,5}
  \draw (1pt,\y cm) -- (-1pt,\y cm) node[anchor=east]
        {{\small $\y$}};  
      \draw[dashed] (5.2,5) -- (2,5) -- (2,1) -- (1,1)
                   -- (1,2) -- (5,2) -- (5,5.2);
      \draw[line width=0.7pt] plot[variable=\t, domain=0.8945:5.2,
        samples=101] ({(\t,{(3*\t - sqrt(5*\t*\t-4))/2})});      
      \draw[line width=0.7pt] plot[variable=\t, domain=0.8945:2.1,
        samples=101] ({(\t,{(3*\t + sqrt(5*\t*\t-4))/2})});   
      \end{tikzpicture}
  \end{center}
  \caption{Vieta jumping on $x^2 - 3xy + y^2 = -1$ in the first quadrant.}
\end{figure}
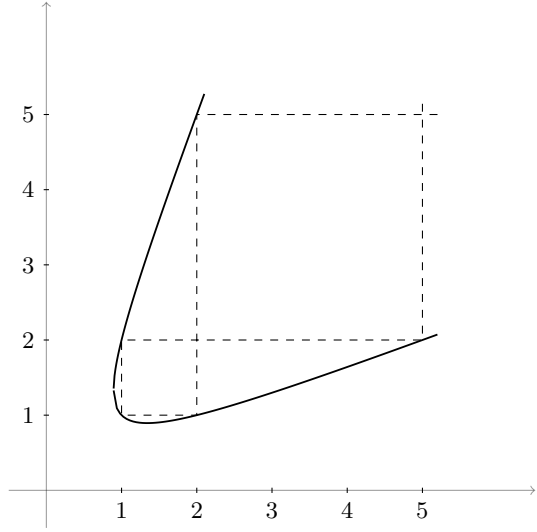

\begin{proof}
  Write $ \frac{x^2 + y^2}{3xy - 1} = k$; then $x^2 - 3kxy + y^2 = -k$.
  If $k \ge 2$, there are no integral solutions by Theorem~\ref{Thpq}.
  If $k = -1$, Vieta jumping yields two sequences of integral points;
  one in the first quadrant, namely
  $$  \ldots \lra (5,2) \lra (2,1)  \lra  (1,1) \lra (1,2)
             \lra (2,5) \lra (5,13) \lra  \ldots, $$
  and the corresponding sequence with negative coordinates in the
  third quadrant.
\end{proof}


For the curve $x^2 - 3xy + y^2 = 1$, Vieta jumping yields two
sequences of integral points:
\begin{align*}
  {} \ldots & \lrs (-3,-8) \lrf (-3,-1) \lrs (0,-1) \lrf (0,1) \lrs (3,1)
           \lrf (3,8) \lrs \ldots \\
  {} \ldots & \lrf (-8,-3) \lrs (-1,-3) \lrs (-1,0) \lrf (1,0) \lrs (1,3)
           \lrs (8,3) \lrf \ldots 
\end{align*}
Also observe that there are many solutions with negative $y$; in this
case, the quotient is the negative of a square. For $k = -4$, for
example, integral solutions of $x^2 + 12xy + y^2 = 4$ are provided by
Vieta jumping from $(2,0)$:
$$ (2,0) \lrs (2,-24) \lrf (286,-24) \lrs \cdots $$

\section{The role of the Pell conic}

In our first solution of the IMO problem (as well as our result concerning
(\ref{M1})) we did not use any algebraic number theory. Our next goal is
explaining a connection of the IMO problem with the arithmetic of certain
real quadratic number fields. We begin with describing the role of the Pell
equation.

It is well know that rational points on conics form a group; if we
fix a rational point $N$, then the sum $P \oplus Q$ of two points
on a conic is the second point of intersection of the parallel to
the line $PQ$ through $N$.

In the special case $k = 4$, the integral points $(a,b)$ on $\cC_4$
correspond directly to units: If $a^2 - 4ab + b^2 = 4$ for integers
$a$ and $b$, then $a$ and $b$ are even, say $a = 2A$ and $b = 2B$;
but then $(a-2b)^2 - 3b^2 = 4$ implies $(A-2B)^2 - 3B^2 = 1$,
hence $(A-2B,B)$ is an integral point on the Pell conic $\cP: x^2 - 3y^2 = 1$.
Conversely, if $(x,y) \in \cP(\Z)$, then $(2x+4y, 2y) \in \cC_4(\Z)$
is an integral point on $\cC_4$. Thus we obtain the two sequences
of integral points on $\cC_4$ given in Table~\ref{TC4}.

\begin{table}[ht!]
$$ \begin{array}{r|c|rr}
  \rsp  k &  \eps^k & a & b \\ \hline
  \rsp -2 &  7 -  4 \sqrt{3} & -2 & -8 \\
  \rsp -1 &  2 -    \sqrt{3} &  0 & -2 \\
  \rsp  0 &     1            &  2 &  0 \\
  \rsp  1 &  2 +    \sqrt{3} &  8 &  2 \\
  \rsp  2 &  7 +  4 \sqrt{3} & 30 &  8 \\
  \rsp  3 & 26 + 15 \sqrt{3} & 82 & 30  
\end{array} \qquad \begin{array}{r|c|rr}
  \rsp  k &  -\eps^k & a & b \\ \hline
  \rsp -2 &  -7 +  4 \sqrt{3} & -30  &   8 \\
  \rsp -1 &  -2 +    \sqrt{3} &  -8  &   2 \\
  \rsp  0 &     -1            &  -2  &   0  \\
  \rsp  1 &  -2 -    \sqrt{3} &   0  &  -2  \\
  \rsp  2 &  -7 -  4 \sqrt{3} &   2  &  -8  \\
  \rsp  3 & -26 - 15 \sqrt{3} &   8  & -30   
  \end{array} $$
  \caption{Integral points on $\cC_4$}\label{TC4}
\end{table}

Since there is no obvious choice of a neutral element $N$ on
$\cC_k: x^2 - kxy + y^2 = k$ it is better to regard $\cC_k$ as
a principal homogeneous space with respect to the associated Pell conic.
We will now explain what we mean by this.

By completing the square we can write $\cC_k$ in the form
$$ (x - \kappa y)^2 - (\kappa^2-1)y^2 = k, $$
where $k = 2\kappa$ (we do not assume that $k$ is even, so $\kappa$
may be a half-integer). This equation tells us that the element
$$ \alpha = x-\kappa y + y\sqrt{m} \in \cO_k $$
with $m = \kappa^2 - 1$ has norm $k$ in the quadratic number field
$k = \Q(\sqrt{m}\,)$, where $m = \kappa^2-1$. The fundamental unit
of $\Q(\sqrt{m}\,)$ is $\eps = \kappa + \sqrt{m}$; in fact we have
$$ \eps = \frac{k + 2 \sqrt{m}}2 = \frac{k + \sqrt{k^2-4}}2 $$
if $k$ is odd; if $k$ is even then $\eps \in \Z[\sqrt{m}\,]$ anyway.
Since $\eps$ has norm $+1$, we find that all elements
$\beta = \pm \eps^f \alpha$ have the same norm $k$ as $\alpha$ and
thus correspond to points on $\cC_k$.

For making this action explicit assume that $(a,b)$ is a point on $\cC_k$.
This point corresponds to the element $\alpha = a - \kappa b + b\sqrt{m}$
with norm $k$; multiplication by $\eps$ gives us the element
$$ \alpha \eps = b + a \kappa + a\sqrt{m}, $$
which corresponds to the point $(ka-b,a)$ on $\cC_k$.

Thus given an integral point $Q$ on the conic $\cC_k$ and an integral
point $P$ on the associated Pell conic $\cP$, there is an integral point
$Q' = Q \oplus P$ on $\cC_k$.

\medskip\noindent {\bf Remark.}
For $k = 4$, the points on the Pell conic $t^2 - 3u^2 = 1$ act on the
points on $\cC_4$:
\begin{align*}
  [(x-2y) + y\sqrt{3} \,] (\cdot 2 + \sqrt{3}\,)
    & = 2(x-2y) + 3y + (2y+x-2y) \sqrt{3} \\
    & = 2x - y + x \sqrt{3} = x' - 2y' + y' \sqrt{3} 
\end{align*}
with
$$ x' = 4x-y \quad \text{and} \quad y' = x. $$
Thus the point $P(x, y)$ on $\cC_r$ is mapped to $Q(4x-y,x)$;
for example, $(2,0)$ gets mapped to $(8,2)$.

\subsection*{Pell equations and Vieta jumping}

Solving Pell equations with negative discriminant is trivial;
these equations have only the trivial solutions corresponding to the
units $\pm 1$, with the exception of $x^2 + y^2 = 1$, whose solutions
correspond to the fourth roots of unity in $\Q(i)$, and the
equation $x^2 - xy + y^2 = 1$: This equation has six solutions
corresponding to the six roots of unity in $\Q(\sqrt{-3}\,)$,
and these six integral points on the ellipse $x^2 - xy + y^2 = 1$
are connected by Vieta jumping (see Fig.~\ref{AEll}).

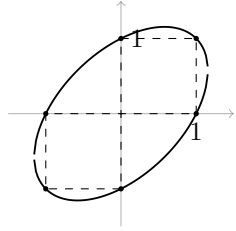
\begin{figure}[ht!]
  \begin{center}
    \begin{tikzpicture}
      \draw[->,opacity=0.3] (-1.5,0) -- (1.5,0);
      \draw[->,opacity=0.3] (0,-1.5) -- (0,1.5);
      \draw ( 1,1pt) -- ( 1,-1pt);
      \draw (-1,1pt) -- (-1,-1pt);
      \draw (-1pt, 1) -- (1pt,1);
      \draw (-1pt,-1) -- (1pt,-1);
      \fill ( 1, 0) circle (1pt);
      \fill ( 1, 1) circle (1pt);
      \fill ( 0, 1) circle (1pt);
      \fill ( 0,-1) circle (1pt);
      \fill (-1, 0) circle (1pt);
      \fill (-1,-1) circle (1pt);
      \draw (1,0) node[anchor=north] {$1$};
      \draw (0,1) node[anchor= west] {$1$};
      \draw [dashed]
           (1,0) -- (1,1) -- (0,1) -- (0,-1) -- (-1,-1) -- (-1,0) -- cycle;
      \draw[line width=0.7pt] plot[variable=\t, domain=-1.154:1.154,
        samples=101] ({(\t,{(\t+sqrt(4-3*\t*\t))/2})});      
      \draw[line width=0.7pt] plot[variable=\t, domain=-1.154:1.154,
        samples=101] ({(\t,{(\t-sqrt(4-3*\t*\t))/2})});      
    \end{tikzpicture}
  \end{center}
  \caption{Integral points on $x^2 - xy + y^2 = 1$}\label{AEll}
\end{figure}


The integral points on the Pell conics $x^2 - 2y^2 = \pm 1$ may also be
obtained using Vieta jumping; all we have to do is write these conics
in the form $x^2 - 2xy - y^2 = \pm 1$.

A parametrization of $x^2 - 2xy - y^2 = 1$ is given by
$$ x = \frac{t^2+1}{t^2+2t-1}, \quad
   y = \frac{2t^2-2t}{t^2+2t-1}. $$

Similarly, the integral points on $(x+y)^2 - 3y^2 = 1$ can be obtained
by Vieta jumping and yield the solutions of the Pell conic $X^2 - 3Y^2 = 1$.   
Integral points on $X^2 - 7Y^2 = 1$ can be found by Vieta jumping after
rewriting the Pell equation in the form $(x+3y)^2 - 7y^2 = 1$.
As a matter of fact, every Pell equation $x^2 - my^2 = 1$ 
with fundamental solution $(x,y) = (t,u)$ can be written in a form
such that the integral solutions can be obtained by Vieta jumping:
Just use a shear mapping to move the point $(t,u)$ to $(1,u)$.

\begin{figure}[ht!]
\begin{tikzpicture}[scale=0.5]
  \clip (-5.5,-3.5) rectangle (5.5,3.5);  
  \draw[->] (-5.5,0) -- (5.5,0);
  \draw[->] (0,-3.5) -- (0,3.5);  
  \foreach \x in {-5,-4,-3,-2,-1,1,2,3,4,5}
  \draw (\x cm,1pt) -- (\x cm,-1pt) node[anchor=north]
        {{\footnotesize $\x$}};
  \foreach \y in {-3,-2,-1,1,2,3}
  \draw (1pt,\y cm) -- (-1pt,\y cm) node[anchor=east]
        {{\footnotesize $\y$}};
  \draw[line width=0.7pt] plot[variable=\t, domain=-2:0.4,samples=101]
  ({(\t*\t+1)/(\t*\t+2*\t-1)},{(-2*\t*\t+2*\t)/(\t*\t+2*\t-1)});
  \draw[line width=0.7pt] plot[variable=\t, domain=-2:0.4,samples=101]
  ({-(\t*\t+1)/(\t*\t+2*\t-1)},{(2*\t*\t-2*\t)/(\t*\t+2*\t-1)});
  \fill ( 1, 0) circle(2.5pt);
  \fill ( 1,-2) circle(2.5pt);
  \fill (-1, 0) circle(2.5pt);
  \fill (-5,-2) circle(2.5pt);
  \fill ( 5, 2) circle(2.5pt);
  \fill (-1, 2) circle(2.5pt);
  \draw [line width=0.8pt, dotted]
  (5,3) -- (5,2) -- (-1,2) -- (-1,0) -- (1,0) -- (1,-2) -- (-5,-2) -- (-5,-3);
\end{tikzpicture} \quad \begin{tikzpicture}[scale=0.5]
  \clip (-5.5,-3.5) rectangle (5.5,3.5);  
  \draw[->] (-5.5,0) -- (5.5,0);
  \draw[->] (0,-3.5) -- (0,3.5);  
  \foreach \x in {-5,...,-1}
  \draw (\x cm,1pt) -- (\x cm,-1pt) node[anchor=north]
        {{\small $\x$}};  
          \foreach \x in {1,...,5}
  \draw (\x cm,1pt) -- (\x cm,-1pt) node[anchor=north]
        {{\small $\x$}};  
  \foreach \y in {-3,-2,-1,1,2,3}
  \draw (1pt,\y cm) -- (-1pt,\y cm) node[anchor=east]
        {{\small $\y$}};
  \draw[line width=0.7pt] plot[variable=\t, domain=-2:0.4,samples=101]
  ({(-2*\t*\t+2*\t)/(\t*\t+2*\t-1)},{(\t*\t+1)/(\t*\t+2*\t-1)});
  \draw[line width=0.7pt] plot[variable=\t, domain=-2:0.4,samples=101]
  ({(2*\t*\t-2*\t)/(\t*\t+2*\t-1)},{-(\t*\t+1)/(\t*\t+2*\t-1)});
  \fill ( 0, 1) circle(2.5pt);
  \fill ( 2,-1) circle(2.5pt);
  \fill ( 0,-1) circle(2.5pt);
  \fill (-2, 1) circle(2.5pt);
  \draw [line width=0.8pt, dotted]
  (3,5) -- (2,5) -- (2,-1) -- (0,-1) -- (0,1) -- (-2,1) -- (-2,-5) -- (-3,-5);
\end{tikzpicture}
\caption{Vieta jumping on the Pell conic $\cP: x^2 - 2xy - y^2 = 1$ and
  on $\cP^{-}: x^2 - 2xy - y^2 = -1$.}\label{AbbPell} 
\end{figure}
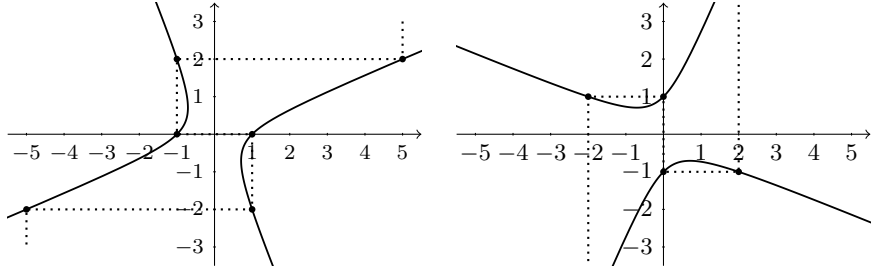

\subsection*{Second Solution of the IMO problem}

We now use the action of the Pell conic on $\cC_k$ to give a second solution
of the IMO problem. The basic idea is to use the action of the Pell
conic to move an integral point to a bounded region close to the origin.
For finding good bounds we use an idea by Davenport published in
\cite{Ankeny}.

\begin{prop}\label{Pr6}
  If the curve $\cC: x^2 - 2kxy + y^2 = \pm m$ for $m \in \N$ has integral
  points, then $m$ is a square or $m \ge 2k+2$.
\end{prop}

The equation $x^2 - 2kxy + y^2 = \pm m$ is equivalent to
\begin{equation}\label{Enorm}
  (x-ky)^2 - (k^2-1)y^2 = \pm m.
\end{equation}
Thus the claim follows from the next lemma, which is due to
Davenport (see \cite{Ankeny}):

\begin{lemma}\label{Lem7}
  Let $x, y, m$ be positive integers such that
  $x^2 - (t^2-1)y^2 = \pm m$. Then $m$ is a square or $|m| \ge 2t+2$.
\end{lemma}

\begin{proof}
  The unit $\eps = t + \sqrt{t^2-1}$ has norm $1$; we have
  $$ (x - y \sqrt{t^2-1}\,)(t + \sqrt{t^2-1}\,)
     = xt - (t^2-1)y + (x-ty) \sqrt{t^2-1}, $$
  hence we have the two equations
  \begin{align*}
    x^2 - (t^2-1)y^2 & = \pm m, \\
    (xt - y(t^2-1))^2 - (x-ty)^2(t^2-1) & = n.
  \end{align*}
  Assume that we have chosen $y > 0$ minimal. Then $|x-ty| \ge y$.
  Thus either $x \ge y(t+1)$ or $0 \le x \le (t-1)y$. This implies
  \begin{align*}
     m & = x^2 - (t^2-1)y^2 \ge (t+1)^2y^2 - (t^2-1)y^2
             = (2t+2)y^2 \ge 2t+2 \quad \text{or} \\
    -m & = x^2 - (t^2-1)y^2 \le (t-1)^2y^2 - (t^2-1)y^2
             = -(2t+2)y^2  \le -(2t+2).
  \end{align*}
  This completes the proof.  
\end{proof}


\section{Elements of small norm in quadratic number fields}

Next we are going to use the arithmetic of quadratic number fields
to give a solution to the IMO problem. The relevant techniques were
studied in \cite{LP} and \cite{Zh}. Equation (\ref{Enorm})
describes elements of small norm in quadratic number fields; in this
section we will prove several results on elements of small norm in
certain real quadratic number fields; these will be used in the
sections below.

We begin by recalling a well known result (going back in a slightly
weaker form to Dirichlet and Chebyshev; see the references in \cite{LQNF}):

\begin{prop}\label{th2}
Let $k = \Q(\sqrt{m}\,)$ be a real quadratic number field,
$\eps > 1$ a unit in k, and $0 \not= \nu = |N\xi\,|$
for $\xi \in k$. Then there is a unit $\eta=\eps^{j}$
such that $\xi \eta = a+b\sqrt{m}$ and 
$$ |a| < \frac{\sqrt{\nu}}{2} \cdot B
   \quad \text{ and } \quad 
   |b| < \frac{\sqrt{\nu}}{2 \sqrt{m}}  \cdot B, 
   \quad \text{where} \quad B = \sqrt{\eps} +1/\sqrt{\eps}. $$
\end{prop}

For $\alpha \in k$ let $\alpha '$ denote the conjugate of $\alpha $.
Thus $N \xi = \xi \xi'$, where $N = N_{k/\Q}$ denotes the absolute norm.

Now let $ \xi \in k$ and a real positive number $c$ be given;
we can find a unit $ \eta \in  E_{k}$ such that
\begin{equation}\label{E1}
   c \leq |\,\xi \eta\,| < c \eps.
\end{equation}
With $\nu = |N\xi|$ we find $ |\xi' \eta'| = \nu/|\xi \eta|$,
so equation (\ref{E1}) yields
\begin{equation}
  \frac{\nu}{c \eps} \leq |\,\xi' \eta'\,| < \frac{\nu}{c}.  
\end{equation}
Writing $ \alpha = \xi \eta = a + b \sqrt{m}$ gives
\begin{align*}
  |2a| & = |\alpha + \alpha '| \leq |\alpha|+|\alpha '|
         < c \eps + \frac{\nu}{c}, \quad \text{and} \\
  |2b\sqrt{m}\,| & = |\alpha -\alpha '| \leq |\alpha|+|\alpha'|
         < c\eps + \frac{\nu}{c}.
\end{align*}
Because we want the coefficients $a$ and $b$ to be as small
as possible, we have to choose $c$ in such a way that
$ c\eps + \nu/c $ becomes a minimum. Putting $c = \sqrt{\nu/\eps} $ we get
\begin{equation}
|2a| < 2\sqrt{\nu\eps},\quad |2b\sqrt{m}\,| < 2\sqrt{\nu\eps}.
\end{equation}
Making use of a lemma due to Cassels, we can improve these bounds:

\begin{lemma}\label{th1}
Suppose that the positive real numbers $x, y$ satisfy the
inequalities $x \leq s$, $y \leq s$, and $xy \leq t$. Then
$x+y \leq s+t/s$.
\end{lemma}

\begin{proof}
  We have $0 \leq (x-s)(y-s) = xy-s(x+y)+s^{2} \leq s^{2}+t-s(x+y)$.
\end{proof}

Putting $x=|\alpha|$ and $y=|\alpha'|$ in Lemma \ref{th1}
we find
\begin{align*}
  |2a| & \leq |\alpha|+|\alpha'|< \sqrt{\nu\eps}+ \sqrt{\nu/\eps},
  \quad \text{and} \\
  |2b \sqrt{m}\,| & < \sqrt{\nu\eps}+\sqrt{\nu/\eps}.
\end{align*}
This proves our claims.

\begin{prop}\label{th4}
  Let $m$ and $n \ge 2$  be natural numbers such that $m = n^2-1$; if
  the diophantine equation $|x^{2}-my^{2}| = \nu$ has solutions in
  $\Z$ for some integer $\nu$ with $0 < \nu < 2n-2$, then $\nu$ is a
  perfect square.
\end{prop}

\begin{proof}
Let $\xi = x+y\sqrt{m}$; then $|\,N \xi\,|=N$, and since
$\eps = n + \sqrt{m}\,>1$ is a unit in $\Z[\sqrt{m}\,]$,
we can find a power $\eta$ of $\varepsilon$ such that
$\xi \eta=a+b\sqrt{m}$ has coefficients $a, \ b$ that satisfy
the bounds in Prop.~\ref{th2}. Since the bound is a monotone decreasing
function of $n$, it attains its maximum at $n = 2$. Since $b$ is an integer,
we must have $|b| \le 1$.

If $b=0$, then $|N\xi|=a^{2}$ is a square. Assume therefore that
$b=\pm 1$; this yields $\alpha = \xi \eta = a \pm \sqrt{m}.$
Now $|N\xi| = |N\alpha| = |a^{2}-m|$ is minimal for
values of $a$ near $\sqrt{m}$, and we find
$$ |a^{2}-m| = \begin{cases}
  2n-2 & \text{ if } a=n-1; \\
   1   & \text{ if } a=n;  \\
  2n+2 & \text{ if } a=n+1. \end{cases} $$
This proves the claim.
\end{proof}

\begin{prop}\label{th4b}
  Let $n \ge 7$ be an odd integer and set $m = n^2 - 4$.
  If $|x^2 - my^2| = \nu$ has a solution in integers with $N < n+2$,
  then either $\nu = n-2$ or $\nu$ is a perfect square.
\end{prop}

The proof is left to the readers.
Observe that $\eps = \frac{n + \sqrt{m}}{2}$ is a unit with norm $+1$;
nontrivial elements with small norm are
$$ \begin{array}{c|cccc}
  \rsp \alpha  & \eps+1 & \eps-1 & 2\eps+1 & 2\eps-1 \\ \hline 
  \rsp N\alpha & 2 + n  &  2 - n & 5 + 2n & 5 - 2n 
  \end{array} $$

\begin{prop}\label{th5}
  Let $m, n, t$ be natural numbers such that $m=n^2+2$ and $n \ge 5$;
  if the diophantine equation $|x^2 - my^2| = \nu$ has solutions in $\Z$
  with $\nu < 2n+1$, then either $\nu = 2n-1$, or $\nu$ or $2 \nu$ is a
  perfect square.
\end{prop}

Here $\delta = n + \sqrt{n^2+2}$ has norm $N\delta  = -2$, and
$\delta^2 = 2\eps$ for the fundamental unit $\eps = n^2+1 + n\sqrt{n^2+2}$.
Since $\eps \approx 2n^2+2$, the bounds in Prop.~\ref{th2}   are not good enough 
for our proof to work. Observe that $\sqrt{\eps} + \frac1{\sqrt{\eps}} = \sqrt{2n^2 +4}$ 
and $\sqrt[4]{\eps} + \frac1{\sqrt[4]{\eps}} = \sqrt{\sqrt{2n^2 + 4} + 2}$.

\begin{proof}
Proof.  Let $\xi= x+y\sqrt{m}, \nu=|N\xi|$,
and suppose that neither $\nu$ nor $2\nu$ are perfect squares.
Letting $\delta = t + \sqrt{m}$, we find $\delta^2 = 2\eps$,
where $\eps$ is a unit in $\Z[\sqrt{m}\,]$.
Obviously we can find a power $\eta$ of $\eps$ such that
\[\sqrt{\nu \sqrt{\eps}}/\eps \le |\xi\eta|  < \sqrt{\nu\sqrt{\eps}}.\]
Now we distinguish two cases:
\begin{itemize}
\item[1.] $\quad \sqrt{\nu/\sqrt{\eps}} \le |\xi \eta| <  \sqrt{\nu \sqrt{\eps}}:$ 
  Writing $\xi\eta = a+b\sqrt{m}$, we find that $|b| \le 1$. In fact,
  $$ | 2b\sqrt{m}\,| = |\xi\eta  - \xi'\eta'| \le  |\xi\eta\,| + | \xi'\eta'|
                              \le \sqrt{\nu} \cdot B $$  
  with 
  $$ B = \sqrt[4]{\eps} + \frac1{\sqrt[4]{\eps}} = \sqrt{\sqrt{2n^2 + 4} + 2}. $$
 Thus
  $$ |b| \le \frac{\sqrt{2n+1}}{\sqrt{n^2+2}} \cdot \sqrt{\sqrt{2n^2 + 4} + 2}. $$
This bound is monotone decreasing for $n \ge 2$ and $< 2$ for $n \ge 5$.
  If $b=0$, then $b$ is a square, so assume $b=\pm 1$. Here the elements $\alpha$ with minimal norm ar
  $$ \begin{array}{c|ccc}
         \rsp \alpha    & n-1  + \sqrt{n^2+2}   &   n + \sqrt{n^2+2}    & n+1  + \sqrt{n^2+2} \\  \hline
         \rsp N\alpha &  - 2n - 1  &            - 2                &             2n - 1                
      \end{array} $$    
  \item[2.] $\quad \sqrt{\nu\sqrt{\eps}}/\eps   \le |\,\xi \eta\,| < \sqrt{\nu/\sqrt{\eps}}:$ 
   Multiplying $\xi \eta$ with $\delta$ we get
   $$ \sqrt{2\nu/\sqrt{\eps}} \le |\,\xi \eta \delta\,| < \sqrt{2\nu \sqrt{\eps}}. $$
   As in case 1. above, we find $|b| \le 1$, where
   $\xi \eta \delta = a+b\sqrt{m}$; if $b=0$, then
   $2\nu=|N(\xi \eta \delta)|=a^{2}$ is a square. If $b=\pm 1$, then
   $a^{2}-mb^{2}$ must be even (because $|N\xi\delta|$ is even), 
   or $|N\xi\delta| \ge 2n-1$.    
\end{itemize}
This completes the proof.
\end{proof}

\subsection*{Third solution using quadratic number fields}

We now translate the IMO problem into a problem about norms
of elements in quadratic number fields.

Assume first that $m = k^2-4$ is even; then so is $k$, and we can write
$k = 2n$ for some integer $n$. If $x^2 - kxy + y^2 = k$, then
$x^2 - 2nxy + y^2 = 2n$ can be written in the form
$(x-ny)^2 - (n^2-1)y^2 = 2n$. By Prop.~\ref{th4}, $k = 2n$ is a
perfect square.

If $k$ is odd, then multiplying  $x^2 - kxy + y^2 = k$ through
by $4$ and completing the square shows that
$$ (2x - ky)^2 - (k^2-4)y^2 = 4k. $$
In this case, $\frac12(2x-ky+y\sqrt{m}\,)$ is an element with norm $k$,
and the claim follows from Prop.~\ref{th4b}.

The proof of the next result uses Prop.~\ref{th5} and is left to the reader:

\begin{prop}
  If there are integers $x, y$ such that $k = \frac{x^2 + 2y^2}{2xy+1}$ is
  an integer, then $k$ is a square or twice a square.
\end{prop}

With computer experiments it is easy to discover a wealth of
problems similar to the original IMO problem. Some of them can
be proved with the methods presented here, others seem to require
new ideas.


\begin{thebibliography}{ACH}

\bibitem{Ankeny}
N. C. Ankeny, S. Chowla, H. Hasse, {\em On the class number of
the real subfield of a cyclotomic field}, J. Reine Angew. Math.
{\bf 217} (1965), 217--220
%

\bibitem{Hasse}
H. Hasse, {\em \"{U}ber mehrklassige, aber eingeschlechtige
reell-quadratische Zahlk\"{o}rper}, El. Math. {\bf 20} (1965), 49--59
%

\bibitem{LQNF} F. Lemmermeyer,
  {\em Quadratic Number Fields},
  Springer 2021
  %
  
\bibitem{LP} F. Lemmermeyer, A. Peth\"o,
  {\em Simplest cubic fields},
  Manuscripta Math {\bf 88} (1995), 53--58
  %

\bibitem{Milchev} V. Milchev,
  {\em Linear recurrence sequences and polynomial division in number theory},
  Crux Math. {\bf 44} (2018), 297--301
  %
  
\bibitem{Rick} J. Rickards,
  {\em Polynomial division in number theory}, \\
  \url{https://sites.google.com/site/imocanada/winter-camps/2017-winter-camp}
  %

\bibitem{Stevens} J. Stevens,
  {\em Olympiad number theory through challenging problems}, \\
  \url{https://s3.amazonaws.com/aops-cdn.artofproblemsolving.com/resources/articles/olympiad-number-theory.pdf}
  %
  
\bibitem{Zh} Xian-ke Zhang,
  {\em Determination of solutions and solvabilities of
    diophantine equations and quadratic fields},
  Algebraic geometry and algebraic number theory,
  World Scientific 1992,p. 189
  %

\bibitem{Zh2} Xian-ke Zhang,
  {\em Solutions of the diophantine equations related to real
    quadratic fields}, Chin. Sci. Bull {\bf 37} (1992), 885--889
  %
  

\end{thebibliography}
\end{document}